\documentclass[portuges,12pt,letter]{article}
\usepackage[centertags]{amsmath}
\usepackage{amsfonts}
\usepackage{newlfont}
\usepackage{amscd}
\usepackage{graphics}
\usepackage{epsfig}
\usepackage{indentfirst}
\usepackage{amsxtra}
\usepackage[latin1]{inputenc}
\usepackage{amssymb, amsmath}
\usepackage{amsthm}
\usepackage[mathscr]{eucal}

\newtheorem{thm}{Theorem}[section]

\newtheorem{remark}[thm]{Remark}

\author{Fabio Silva Botelho \\
Department of Mathematics \\
Federal University of Santa Catarina - UFSC \\
Florian\'{o}polis, SC - Brazil}
\date{}
\title{\bf  A convex dual variational formulation for non-convex optimization applied to a non-linear model of plates }

\begin{document}
\maketitle

\abstract{This article develops   duality principles applicable to the non-linear Kirchhoff-Love model of plates.
The  results are obtained through standard tools of convex analysis, functional analysis, calculus of variations and duality theory. The main duality principle concerns a convex (in fact concave) dual variational formulation and related new optimality conditions for the model in question. Finally, in the last section we develop some global existence results for a similar model in elasticity.}

\section{Introduction}
 In this article, in a first step,  we develop  a new existence proof and  a dual variational formulation for the Kirchhoff-Love thin plate model. Previous results on existence in mathematical elasticity and related models may be found in \cite{903,[3],[4]}.

At this point we refer to the exceptionally important article "A contribution to contact problems for a class of solids and structures" by
W.R. Bielski and J.J. Telega, \cite{85},  published in 1985,  as the first one to successfully  apply and generalize the convex analysis approach to a model in non-convex and non-linear mechanics.

The present work is, in some sense, a kind of extension  of this previous work \cite{85} and others such as \cite{2900}, which greatly influenced and
inspired my work and recent books \cite{120,700}.

Here we highlight that such earlier results establish the complementary energy under the hypothesis of positive definiteness of the membrane force tensor at a critical point (please see  \cite{85,2900} for details).

We have obtained a dual variational formulation which allows the global optimal point in question not to be positive definite (for  related results see F.S. Botelho \cite{120}), but also not necessarily negative definite. The approach developed also includes sufficient conditions of optimality for the primal problem.
It is worth mentioning that the standard tools of convex analysis used in this text may be found in \cite{[6],120}, for example.

At this point we start to describe the primal formulation.

    Let $\Omega\subset\mathbb{R}^2$ be an open, bounded, connected set  which  represents the middle surface of a plate
of thickness $h$. The boundary of $\Omega$, which is assumed to be regular (Lipschitzian), is
denoted by $\partial \Omega$. The vectorial basis related to the cartesian
system $\{x_1,x_2,x_3\}$ is denoted by $( \textbf{a}_\alpha,
\textbf{a}_3)$, where $\alpha =1,2$ (in general Greek indices stand
for 1 or 2), and where $\textbf{a}_3$ is the vector normal to $\Omega$, whereas $\textbf{a}_1$ and $\textbf{a}_2$ are orthogonal vectors parallel to $\Omega.$  Also,
$\textbf{n}$ is
the outward normal to the plate surface.

    The displacements will be denoted by
$$ \hat{\textbf{u}}=\{\hat{u}_\alpha,\hat{u}_3\}=\hat{u}_\alpha
\textbf{a}_\alpha+ \hat{u}_3 \textbf{a}_3.
$$

The Kirchhoff-Love relations are
\begin{eqnarray}&&\hat{u}_\alpha(x_1,x_2,x_3)=u_\alpha(x_1,x_2)-x_3
w(x_1,x_2)_{,\alpha}\; \nonumber \\ && \text{ and }
\hat{u}_3(x_1,x_2,x_3)=w(x_1,x_2).
\end{eqnarray}

Here $ -h/2\leq x_3 \leq h/2$ so that we have
$u=(u_\alpha,w)\in U$ where
\begin{eqnarray}
U&=&\left\{(u_\alpha,w)\in W^{1,2}(\Omega; \mathbb{R}^2) \times
W^{2,2}(\Omega) , \right.\nonumber \\ &&\; u_\alpha=w=\frac{\partial w}{\partial \textbf{n}}=0 \left.
\text{ on } \partial \Omega \right\} \nonumber \\ &=&W_0^{1,2}(\Omega; \mathbb{R}^2) \times W_0^{2,2}(\Omega).\nonumber\end{eqnarray}
It is worth emphasizing that the boundary conditions here specified refer to a clamped plate.

We define the operator $\Lambda: U \rightarrow Y \times Y$, where $Y=Y^*=L^2(\Omega; \mathbb{R}^{2 \times 2})$, by
$$\Lambda(u)=\{\gamma(u), \kappa(u)\},$$
$$\gamma_{\alpha\beta}(u)= \frac{u_{\alpha,\beta}+u_{\beta,\alpha}}{2}+\frac{w_{,\alpha} w_{,\beta}}{2},$$
$$\kappa_{\alpha \beta}(u)=-w_{,\alpha \beta}.$$
The constitutive relations are given by
\begin{equation}
N_{\alpha\beta}(u)=H_{\alpha\beta\lambda\mu} \gamma_{\lambda\mu}(u),
\end{equation}
\begin{equation}
M_{\alpha\beta}(u)=h_{\alpha\beta\lambda\mu} \kappa_{\lambda\mu}(u),
\end{equation}
where: $$\{H_{\alpha\beta\lambda\mu}\}=\left\{h \left(\frac{4\lambda^h \mu^h}{\lambda^h+2\mu^h}\delta_{\alpha\beta}\delta_{\lambda \mu}+2\mu^h(\delta_{\alpha \lambda} \delta_{\beta\mu}+\delta_{\alpha \mu} \delta_{\beta\lambda})\right)\right\}$$
 and
 $$\{h_{\alpha\beta\lambda\mu}\}=\frac{h^2}{3}\{H_{\alpha\beta\lambda\mu}\}$$
 are symmetric positive definite   fourth order tensors. Here $\lambda^h>0$ and $\mu^h>0$ are the Lam\'{e} constants which depend on $h$ and $\{\delta_{\alpha\beta}\}$ is the Kronecker delta.

 From now on, in an appropriate sense, we denote $\{\overline{H}_{\alpha\beta \lambda \mu}\}=\{H_{\alpha\beta \lambda \mu}\}^{-1}$ and $\{\overline{h}_{\alpha\beta \lambda \mu}\}=\{h_{\alpha\beta \lambda \mu}\}^{-1}$.

 Furthermore
 $\{N_{\alpha\beta}\}$ denote the membrane force tensor and
 $\{M_{\alpha\beta}\}$ the moment one.
    The plate stored energy, represented by $(G\circ
    \Lambda):U\rightarrow\mathbb{R}$ is expressed by
 \begin{equation}\label{80} (G\circ\Lambda)(u)=\frac{1}{2}\int_{\Omega}
  N_{\alpha\beta}(u)\gamma_{\alpha\beta}(u)\;dx+\frac{1}{2}\int_{\Omega}
  M_{\alpha\beta}(u)\kappa_{\alpha\beta}(u)\;dx
  \end{equation}
 and the external work, represented by $F:U\rightarrow\mathbb{R}$, is given by
   \begin{equation}\label{81} F(u)=\langle w,P \rangle_{L^2(\Omega)}+\langle u_\alpha,P_\alpha \rangle_{L^2(\Omega)}
,\end{equation} where $P, P_1, P_2 \in L^2(\Omega)$ are external loads in the directions $\textbf{a}_3$, $\textbf{a}_1$ and $\textbf{a}_2$ respectively. The potential energy, denoted by
$J:U\rightarrow\mathbb{R}$ is expressed by:
$$
J(u)=(G\circ\Lambda)(u)-F(u)
$$

Finally, we also emphasize from now on, as their meaning are clear, we may denote $L^2(\Omega)$ and $L^2(\Omega; \mathbb{R}^{2 \times 2})$ simply by $L^2$, and the respective norms by $\|\cdot \|_2.$ Moreover derivatives are always understood in the distributional sense, $\mathbf{0}$ may denote the zero vector in  appropriate Banach spaces and, the following and relating notations are used:
$$w_{,\alpha\beta}=\frac{\partial^2 w}{\partial x_\alpha \partial x_\beta},$$
$$u_{\alpha,\beta}=\frac{\partial u_\alpha}{\partial x_\beta},$$
$$N_{\alpha\beta,1}=\frac{\partial N_{\alpha\beta}}{\partial x_1},$$
and
$$N_{\alpha\beta,2}=\frac{\partial N_{\alpha\beta}}{\partial x_2}.$$
\section{On the existence of a global minimizer}

At this point we present an existence result concerning the Kirchhoff-Love plate model.

We start with the following two remarks.
\begin{remark}\label{us123W}Let $\{P_\alpha\} \in L^\infty(\Omega;\mathbb{R}^2)$. We may easily obtain by appropriate Lebesgue integration $\{\tilde{T}_{\alpha \beta}\}$ symmetric and such that

 $$\tilde{T}_{\alpha\beta,\beta}=-P_\alpha, \text{ in } \Omega.$$

 Indeed, extending $\{P_\alpha\}$ to zero outside $\Omega$ if necessary, we may set
 $$\tilde{T}_{11}(x,y)=-\int_0^xP_1(\xi,y)\;d\xi,$$
  $$\tilde{T}_{22}(x,y)=-\int_0^yP_2(x,\xi)\;d\xi,$$
  and $$\tilde{T}_{12}(x,y)=\tilde{T}_{21}(x,y)=0, \text{ in } \Omega.$$

 Thus, we may choose a $C>0$ sufficiently big, such that $$\{T_{\alpha\beta}\}=\{\tilde{T}_{\alpha\beta} +C \delta_{\alpha\beta}\}$$ is positive definite $\text{ in } \Omega$, so that

 $$T_{\alpha\beta,\beta}=\tilde{T}_{\alpha\beta,\beta}=-P_\alpha,$$

 where $$\{\delta_{\alpha\beta}\}$$ is the Kronecker delta.

 So, for the kind of boundary conditions of the next theorem, we do NOT have any restriction for the $\{P_\alpha\}$ norm.

 Summarizing, the next result is new and it is really a step forward concerning the previous one in Ciarlet \cite{[3]} and the previous results in \cite{700}. We emphasize this result and
 its proof through such a tensor $\{T_{\alpha\beta}\}$ are new, even though the final part of the proof is established through a standard procedure in
 the calculus of variations.

It is also worth mentioning  in the concerning primal formulation we have included a term denoted by $$J_5(u)=\langle \varepsilon_\alpha, u_\alpha^2\rangle_{L^2(\Gamma_t)}.$$

This term, even in the case where the positive $\varepsilon_\alpha \in L^2(\Gamma)$ is of small magnitude,  has an amazing structural stabling effect on the plate model in question and makes the primal energy functional bounded below. This feature makes viable the proof of this existence result for a more general set of boundary conditions. We highlight such a functional $J_5$ corresponds to adding a distribution of springs on the portion boundary $\Gamma_t$.

 Finally, we also highlight the duality principles and concerning optimality conditions are established through new functionals. Similar results may be found
 in \cite{120,700}. Indeed the results here developed represent some advances concerning those presented in \cite{120} in 2014 and in \cite{700} of 2020.
\end{remark}

\begin{thm}\label{usp10}Let $\Omega \subset \mathbb{R}^2$ be an open, bounded, connected set with a Lipschitzian boundary denoted by $\partial \Omega=\Gamma.$ Suppose $(G \circ \Lambda):U \rightarrow \mathbb{R}$ is defined by
$$G(\Lambda u)=G_1(\gamma (u))+G_2(\kappa (u)), \; \forall u \in U,$$
where
$$G_1(\gamma u)= \frac{1}{2}\int_\Omega H_{\alpha\beta\lambda\mu} \gamma_{\alpha\beta}(u)\gamma_{\lambda \mu}(u)\;dx,$$
and
$$G_2(\kappa u)=\frac{1}{2}\int_\Omega h_{\alpha\beta\lambda\mu}\kappa_{\alpha\beta}(u)\kappa_{\lambda \mu}(u)\;dx,$$
where
$$\Lambda(u)=(\gamma(u),\kappa(u))=(\{\gamma_{\alpha\beta}(u)\},\{\kappa_{\alpha\beta}(u)\}),$$
$$\gamma_{\alpha\beta}(u)=\frac{u_{\alpha,\beta}+u_{\beta,\alpha}}{2}+\frac{w_{,\alpha} w_{,\beta}}{2},$$
$$\kappa_{\alpha\beta}(u)=-w_{,\alpha\beta},$$
where,
 \begin{eqnarray} U&=&\left\{u=(u_\alpha,w)=(u_1,u_2,w) \in W^{1,2}(\Omega; \mathbb{R}^2) \times W^{2,2}(\Omega)\;:
\right. \nonumber \\ && \left.u_\alpha=w= \frac{\partial w}{\partial \mathbf{n}}=0, \text{ on } \Gamma_0\right\}.
\end{eqnarray}
Here $\partial \Omega =\Gamma_0 \cup \Gamma_t$ and the Lebesgue measures $$m_{\Gamma}(\Gamma_0 \cap \Gamma_t)=0,$$
and $$m_\Gamma (\Gamma_0)>0.$$

We also define, \begin{eqnarray}\langle u,\mathbf{f}\rangle_{L^2}&=&\langle w,P\rangle_{L^2(\Omega)}+\langle u_\alpha, P_\alpha \rangle_{L^2(\Omega)} \nonumber \\ &&+\langle P^t_\alpha, u_\alpha \rangle_{L^2(\Gamma_t)} +\langle P^t,w \rangle_{L^2(\Gamma_t)},
\end{eqnarray}
\begin{eqnarray} F_1(u)&=& -\langle w,P\rangle_{L^2(\Omega)}-\langle u_\alpha, P_\alpha \rangle_{L^2(\Omega)}-\langle P^t_\alpha, u_\alpha \rangle_{L^2(\Gamma_t)} \nonumber \\&&
-\langle P^t,w \rangle_{L^2(\Gamma_t)}+\langle \varepsilon_\alpha, u_\alpha^2 \rangle_{L^2(\Gamma_t)} \nonumber \\
&=& -\langle u,\mathbf{f}\rangle_{L^2}+\langle \varepsilon_\alpha, u_\alpha^2 \rangle_{L^2(\Gamma_t)}
\nonumber \\ &\equiv&  -\langle u,\mathbf{f}_1\rangle_{L^2}-\langle u_\alpha, P_\alpha \rangle_{L^2(\Omega)}+\langle \varepsilon_\alpha, u_\alpha^2 \rangle_{L^2(\Gamma_t)} ,
\end{eqnarray}
where  $$\langle u,\mathbf{f}_1\rangle_{L^2}=\langle u,\mathbf{f}\rangle_{L^2}-\langle u_\alpha, P_\alpha \rangle_{L^2(\Omega)},$$ $\varepsilon_\alpha \in L^2(\Gamma)$ is such that $\varepsilon_\alpha>0, \text{ in } \Gamma, \; \forall \alpha \in \{1,2\}$ and
$$\mathbf{f}=(P_\alpha,P) \in L^\infty(\Omega;\mathbb{R}^3).$$

Let $J:U \rightarrow \mathbb{R}$ be defined by
$$J(u)=G(\Lambda u)+F_1(u),\; \forall u \in U.$$
Assume there exists $\{c_{\alpha\beta}\} \in \mathbb{R}^{2 \times 2}$ such that $c_{\alpha\beta}>0,\; \forall \alpha,\beta \in \{1,2\}$ and $$G_2(\kappa(u)) \geq c_{\alpha\beta}\|w_{,\alpha\beta}\|_2^2,\; \forall u \in U.$$

Under such hypotheses, there exists $u_0 \in U$ such that
$$J(u_0)=\min_{u \in U} J(u).$$
\end{thm}
\begin{proof} Observe that we may find $\mathbf{T}_\alpha=\{(T_\alpha)_\beta\}$ such that
$$div \mathbf{T}_{\alpha}=T_{\alpha\beta,\beta}=-P_\alpha$$ an also such that $\{T_{\alpha\beta}\}$ is positive definite and symmetric (please, see Remark \ref{us123W}).

Thus defining \begin{equation}\label{a.1}v_{\alpha\beta}(u)= \frac{u_{\alpha,\beta}+u_{\beta,\alpha}}{2}+\frac{1}{2}w_{,\alpha}w_{,\beta},\end{equation} we obtain
\begin{eqnarray}\label{usp1}
J(u)&=&G_1(\{v_{\alpha\beta}(u)\})+G_2(\kappa(u))-\langle u,\mathbf{f}\rangle_{L^2}+\langle \varepsilon_\alpha, u_\alpha^2 \rangle_{L^2(\Gamma_t)} \nonumber \\ &=& G_1(\{v_{\alpha\beta}(u)\})+G_2(\kappa(u))+\langle T_{\alpha\beta,\beta},u_\alpha \rangle_{L^2(\Omega)}-\langle u,\mathbf{f}_1\rangle_{L^2}+\langle \varepsilon_\alpha, u_\alpha^2 \rangle_{L^2(\Gamma_t)}\nonumber \\ &=& G_1(\{v_{\alpha\beta}(u)\})+G_2(\kappa(u))-\left\langle T_{\alpha\beta},\frac{u_{\alpha,\beta}+u_{\beta,\alpha}}{2} \right\rangle_{L^2(\Omega)}
\nonumber \\ &&+\langle T_{\alpha\beta}n_\beta ,u_\alpha\rangle_{L^2(\Gamma_t)}-\langle u,\mathbf{f}_1\rangle_{L^2}+\langle \varepsilon_\alpha, u_\alpha^2 \rangle_{L^2(\Gamma_t)} \nonumber \\ &=& G_1(\{v_{\alpha\beta}(u)\})+G_2(\kappa(u))-\left\langle T_{\alpha\beta},v_{\alpha\beta}(u)-\frac{1}{2}w_{,\alpha}w_{,\beta} \right\rangle_{L^2(\Omega)}-\langle u,\mathbf{f}_1\rangle_{L^2}+\langle \varepsilon_\alpha, u_\alpha^2 \rangle_{L^2(\Gamma_t)} \nonumber \\ &&+\langle T_{\alpha\beta}n_\beta ,u_\alpha\rangle_{L^2(\Gamma_t)}\nonumber \\ &\geq& c_{\alpha\beta}\|w_{,\alpha\beta}\|_2^2+\frac{1}{2}\left\langle T_{\alpha\beta},w_{,\alpha}w_{,\beta} \right\rangle_{L^2(\Omega)}-\langle u,\mathbf{f}_1\rangle_{L^2}+\langle \varepsilon_\alpha, u_\alpha^2 \rangle_{L^2(\Gamma_t)}+G_1(\{v_{\alpha\beta}(u)\})\nonumber \\ &&-\langle T_{\alpha\beta},v_{\alpha\beta}(u)\rangle_{L^2(\Omega)}
+\langle T_{\alpha\beta}n_\beta ,u_\alpha\rangle_{L^2(\Gamma_t)}.
\end{eqnarray}

From this, since $\{T_{\alpha\beta}\}$ is positive definite,  clearly $J$ is bounded below.

Let $\{u_n\} \in U$ be a minimizing sequence for $J$. Thus there exists $\alpha_1 \in \mathbb{R}$ such that
$$\lim_{n \rightarrow \infty}J(u_n)= \inf_{u \in U} J(u)=\alpha_1.$$

From (\ref{usp1}), there exists $K_1>0$ such that
$$\|(w_n)_{,\alpha\beta}\|_2< K_1,\forall \alpha,\beta \in \{1,2\},\;n \in \mathbb{N}.$$

Therefore, there exists $w_0 \in W^{2,2}(\Omega)$ such that, up to a subsequence not relabeled,
$$(w_n)_{,\alpha\beta} \rightharpoonup (w_0)_{,\alpha\beta},\; \text{ weakly in } L^2,$$
 $\forall \alpha,\beta \in \{1,2\}, \text{ as } n \rightarrow \infty.$

Moreover, also up to a subsequence not relabeled,
\begin{equation}\label{a.2}(w_n)_{,\alpha} \rightarrow (w_0)_{,\alpha},\; \text{ strongly in } L^2  \text{ and } L^4,\end{equation}
 $\forall \alpha, \in \{1,2\}, \text{ as } n \rightarrow \infty.$

Also from (\ref{usp1}), there exists $K_2>0$ such that,
$$\|(v_n)_{\alpha\beta}(u)\|_2< K_2,\forall \alpha,\beta \in \{1,2\},\;n \in \mathbb{N},$$
and thus, from this, (\ref{a.1}) and (\ref{a.2}), we may infer that
there exists $K_3>0$ such that
$$\|(u_n)_{\alpha,\beta}+(u_n)_{\beta,\alpha}\|_2< K_3,\forall \alpha,\beta \in \{1,2\},\;n \in \mathbb{N}.$$

From this and Korn's inequality, there exists $K_4>0$ such that

$$\|u_n\|_{W^{1,2}(\Omega;\mathbb{R}^2)} \leq K_4, \; \forall n \in \mathbb{N}.$$
So, up to a subsequence not relabeled, there exists $\{(u_0)_\alpha\} \in W^{1,2}(\Omega, \mathbb{R}^2),$ such that
$$(u_n)_{\alpha,\beta}+(u_n)_{\beta,\alpha} \rightharpoonup (u_0)_{\alpha,\beta}+(u_0)_{\beta,\alpha},\; \text{ weakly in } L^2,$$
 $\forall \alpha,\beta \in \{1,2\}, \text{ as } n \rightarrow \infty,$
and,
$$(u_n)_{\alpha} \rightarrow (u_0)_{\alpha},\; \text{ strongly in } L^2,$$
 $\forall \alpha \in \{1,2\}, \text{ as } n \rightarrow \infty.$

Moreover, the boundary conditions satisfied by the subsequences are also satisfied for $w_0$ and $u_0$ in a trace sense, so that
 $$u_0=((u_0)_\alpha,w_0) \in U.$$

From this, up to a subsequence not relabeled, we get
$$\gamma_{\alpha\beta}(u_n) \rightharpoonup \gamma_{\alpha\beta}(u_0), \text{ weakly in } L^2,$$
$\forall \alpha,\beta \in \{1,2\},$
and
$$\kappa_{\alpha\beta}(u_n) \rightharpoonup \kappa_{\alpha\beta}(u_0), \text{ weakly in } L^2,$$
$\forall \alpha,\beta \in \{1,2\}.$

Therefore, from the convexity of $G_1$ in $\gamma$ and $G_2$ in $\kappa$ we obtain
\begin{eqnarray}
\inf_{u \in U}J(u)&=& \alpha_1 \nonumber \\ &=& \liminf_{n \rightarrow \infty}J(u_n) \nonumber \\ &\geq&
J(u_0).
\end{eqnarray}

Thus, $$J(u_0)=\min_{u\in U}J(u).$$

The proof is complete.
\end{proof}
\section{The  duality principles}

Considering the statements and definitions of the previous sections, define again the functional $J:U \rightarrow \mathbb{R}$ by
\begin{eqnarray}
J(u)&=&\frac{1}{2}\int_\Omega h_{\alpha\beta\mu\lambda} \kappa_{\alpha\beta}(u) \kappa_{\lambda \mu}(u)\;dx
\nonumber \\ &=&\frac{1}{2}\int_\Omega H_{\alpha\beta\mu\lambda} \gamma_{\alpha\beta}(u) \gamma_{\lambda \mu}(u)\;dx
\nonumber \\ &&-\langle w,P \rangle_{L^2}-\langle u_\alpha,P_\alpha \rangle_{L^2}
\end{eqnarray}
where
$$(u_\alpha,w)=(u_1,u_2,w) \in U=W_0^{1,2}(\Omega, \mathbb{R}^2) \times W_0^{2,2}(\Omega),$$
$P \in L^2(\Omega),\; P_\alpha \in L^2(\Omega)$ for $\alpha \in \{1,2\}.$

Also
$$\kappa_{\alpha\beta}(u)=-w_{,\alpha\beta},$$
$$\gamma_{\alpha\beta}(u)=\frac{1}{2}(u_{\alpha,\beta}+u_{\beta,\alpha})+\frac{1}{2} w_{,\alpha}w_{,\beta}.$$

Here we define $\tilde{G}_1:U \rightarrow \mathbb{R}$ by
$$\tilde{G}_1(u)=G_1(\{w_{,\alpha\beta}\})=\frac{1}{2}\int_\Omega h_{\alpha\beta\mu\lambda} \kappa_{\alpha\beta}(u) \kappa_{\lambda \mu}(u)\;dx,$$
$\tilde{G}_2:U \rightarrow \mathbb{R}$ by
$$\tilde{G}_2(u)=G_2(\{u_{\alpha,\beta},w_{,\alpha}\})=\frac{1}{2}\int_\Omega H_{\alpha\beta\mu\lambda} \gamma_{\alpha\beta}(u) \gamma_{\lambda \mu}(u)\;dx
+\frac{K}{2}\langle w_{,\alpha},w_{,\alpha} \rangle_{L^2},$$
and $F:U \rightarrow \mathbb{R}$ by
$$F(u)=\frac{K}{2}\langle w_{,\alpha},w_{,\alpha} \rangle_{L^2}.$$

Moreover, denoting $Y=Y^*=W^{2,2}_0(\Omega),$ $Y_1=Y_1^*=L^2(\Omega;\mathbb{R}^{2 \times 2}),$ $Y_2=Y_2^*=L^2(\Omega;\mathbb{R}^2)$ and  $Y_3=Y_3^*=Y_1,$ define also $G_1^*:Y_3^* \times Y^* \rightarrow \mathbb{R}$ by
\begin{eqnarray}G_1^*(\tilde{M},z^*)&=&\sup_{ v_3 \in Y_3}\{ \langle (v_3)_{\alpha\beta}, \tilde{M}_{\alpha\beta}+z^* \delta_{\alpha\beta} \rangle_{L^2}-G_1(\{(v_3)_{\alpha\beta}\}) \nonumber \\ &=&
\frac{1}{2}\int_\Omega \overline{h}_{\alpha\beta\lambda\mu}(\tilde{M}_{\alpha\beta}+z^* \delta_{\alpha\beta})(\tilde{M}_{\lambda\mu}+z^* \delta_{\lambda\mu})\;dx,
\end{eqnarray}
$G_2^*:Y_1^* \times Y_2^* \rightarrow \mathbb{R}$ by
\begin{eqnarray}G_2^*(N,Q)&=&\sup_{ (v_1,v_2) \in Y_1 \times Y_2}\{ \langle (v_1)_{\alpha\beta},N_{\alpha\beta}\rangle_{L^2}+\langle (v_2)_\alpha,Q_\alpha \rangle_{L^2}-G_2(\{(v_1)_{\alpha\beta},(v_2)_\alpha\}) \nonumber \\ &=&
\frac{1}{2}\int_\Omega \overline{N^K_{\alpha\beta}}Q_\alpha Q_\beta\;dx+\frac{1}{2}\int_\Omega \overline{H}_{\alpha\beta\lambda\mu}N_{\alpha\beta}N_{\lambda\mu}\;dx,
\end{eqnarray}
in $B^*$ where
$$B^*=\{v^*=(N,Q,M) \in Y_1^*\times Y_2^* \times Y_3^*\;:\; \{N_{\alpha\beta}+K \delta_{\alpha\beta}\} \text{ is positive definite in } \Omega\}.$$
Here we have denoted $$\{\overline{N^K_{\alpha\beta}}\}=\{N_{\alpha\beta}+K\delta_{\alpha\beta}\}^{-1}.$$
Finally,  define

$F^*:Y^* \rightarrow \mathbb{R}$ by
\begin{eqnarray}F^*(z^*)&=&\sup_{u \in U}\{ -\langle \nabla w,\nabla z^* \rangle_{L^2}-F(u)\}
\nonumber \\ &=& \frac{1}{2K}\int_\Omega  |\nabla z^*|^2\;dx.\end{eqnarray}

Furthermore, define $$J^*: B^* \times Y^* \rightarrow \mathbb{R},$$ by
$$J^*(N,Q,\tilde{M},z^*)=-G_1^*(\tilde{M},z^*)-G_2^*(N,Q)+F^*(z^*),$$
$$J_1^*:B^* \times Y^* \rightarrow \mathbb{R}$$ by
\begin{eqnarray}J_1^*(N,Q,\tilde{M},z^*)&=& J^*(N,Q,\tilde{M},z^*) \nonumber \\ &&+
\frac{1}{2}\int_\Omega [C_0(\overline{h}_{11\lambda \mu}(\tilde{M}_{\lambda \mu}+z^*\delta_{\lambda \mu})_{,22}-\overline{h}_{22\lambda \mu}(\tilde{M}_{\lambda \mu}+z^*\delta_{\lambda \mu})_{,11})] \nonumber \\ && \times(\overline{h}_{11\lambda \mu}(\tilde{M}_{\lambda \mu}+z^*\delta_{\lambda \mu})_{,22}-\overline{h}_{22\lambda \mu}(\tilde{M}_{\lambda \mu}+z^*\delta_{\lambda \mu})_{,11})\;dx,
\end{eqnarray}
where $$C_0=(1-\varepsilon_3)(\overline{h}_{2222}D_{1111}-2\overline{h}_{1122}D_{11}D_{22}+\overline{h}_{1111}D_{2222})^{-1}$$ and
$0<\varepsilon_3<1$.

We suppose $\varepsilon_3$ is sufficiently close to $1$ so that $J_1^*$ is concave in $v^*=(N,Q,\tilde{M}),$ on $B^*$ (in fact, for this kind of tensor $\{h_{\alpha\beta\lambda\mu}\},$ through an analysis of the Hessian in question with the help of the softwares MATHEMATICA or MAPLE, we may infer that $J_1^*$ is concave in $v^*$ on $B^*$ for any value $0<\varepsilon_3<1).$

Here we  remark to generically denote for $y \in L^2(\Omega),$ $$(\overline{h}_{2222}D_{1111}-2\overline{h}_{1122}D_{11}D_{22}+\overline{h}_{1111}D_{2222})^{-1}(y)= w,$$ if and only if $$w \in W_0^{2,2}(\Omega)$$ and $$(\overline{h}_{2222}D_{1111}-2\overline{h}_{1122}D_{11}D_{22}+\overline{h}_{1111}D_{2222})[w]=y,$$ where also generically we have denoted $$D_{\alpha\beta \gamma \mu} [w]= w_{,\alpha\beta\lambda\mu}=\frac{ \partial ^4 }{\partial x_\alpha\;\partial x_\beta\;\partial x_\lambda \;\partial x_\mu}[w].$$

Moreover, define
$$J_2^*:Y^* \rightarrow \mathbb{R}$$ by

\begin{eqnarray}
J_2^*(z^*)&=& -\frac{1}{2}\int_\Omega \overline{h}_{\alpha\beta\lambda\mu}(z^* \delta_{\alpha\beta})(z^* \delta_{\lambda\mu})\;dx
\nonumber \\ &&+
\frac{1}{2}\int_\Omega [C_0(\overline{h}_{11\lambda \mu}(z^*\delta_{\lambda \mu})_{,22}-(\overline{h}_{22\lambda \mu}z^*\delta_{\lambda \mu})_{,11})] \nonumber \\ && \times(\overline{h}_{11\lambda \mu}(z^*\delta_{\lambda \mu})_{,22}-(\overline{h}_{22\lambda \mu}z^*\delta_{\lambda \mu})_{,11})\;dx \nonumber \\ &&
 +F^*(z^*)
 \end{eqnarray}
and
$J_3^*:B^* \times Y^* \times U \rightarrow \mathbb{R}$ by
\begin{eqnarray}J_3^*(N,Q,\tilde{M},z^*,u)&=&J^*(N,Q,\tilde{M},z^*)
\nonumber \\ &&+\langle w, \tilde{M}_{\alpha\beta, \alpha\beta}-Q_{\alpha, \alpha}-P \rangle_{L^2}
\nonumber \\ &&-\langle u_\alpha, N_{\alpha\beta,\beta}+P_\alpha \rangle_{L^2}
\end{eqnarray}
Also, define
$A^*=A_1 \cap A_2$ where
$$A_1=\{v^*=(N,Q,\tilde{M}) \in B^* \;:\; N_{\alpha\beta,\beta}+P_\alpha=0, \text{ in } \Omega\},$$
$$A_2=\{v^*=(N,Q,\tilde{M}) \in B^* \;:\; M_{\alpha\beta,\alpha\beta}-Q_{\alpha,\alpha}-P=0, \text{ in } \Omega\},$$

With such statements and definitions in mind we establish our main result, which is summarized by the following theorem.

\begin{thm} Denoting $v_0^*=(N_0,Q_0,\tilde{M}_0),$ suppose $(v_0^*,z_0^*,u_0) \in B^* \times Y^*\times U$ is such that
$$\delta J_3^*(v_0^*,z_0^*,u_0)=\mathbf{0}.$$
Under such hypotheses we have, $$\delta J(u_0)=\mathbf{0},$$ $$v_0^*=(N_0,Q_0, \tilde{M}_0) \in A^*$$ and
\begin{eqnarray}
J(u_0)&=& \inf_{ u \in U} \left\{J(u)+\frac{K}{2}\int_\Omega |\nabla w-\nabla w_0|^2\;dx\right\} \nonumber \\ &=&
J^*(v_0^*,z_0^*)  \nonumber \\ &=& \sup_{v^* \in A^*} J^*(v^*,z_0^*) =J_1^*(v_0^*,z_0^*).
\end{eqnarray}
\end{thm}
\begin{proof}
From the hypotheses,
$$\delta J_3^*(v_0^*,z_0^*,u_0)=\mathbf{0}.$$
From the variation in $\tilde{M}$ we have
$$-\overline{h}_{\alpha\beta\lambda\mu}((\tilde{M}_0)_{\lambda\mu}+z_0^* \delta_{\lambda \mu})+(w_0)_{,\alpha\beta}=0, \text{ in } \Omega.$$

Thus,
$$(\tilde{M}_0)_{\alpha\beta}=h_{\alpha\beta\lambda\mu}(w_0)_{\lambda \mu}-z_0^*\delta_{\alpha\beta}.$$

From the variation in $Q$ we obtain

$$-\overline{N^K_{\alpha\beta}} (Q_0)_\beta+(w_0)_{,\alpha}=0, \text{ in } \Omega$$ so that
$$(Q_0)_\alpha =(N_0)_{\alpha\beta}(w_0)_\beta+K (w_0)_\alpha.$$

From the variation in $N$ we get
$$(\overline{N_0^K})_{\alpha\rho}(Q_0)_\rho(\overline{N_0^K})_{\alpha\mu}(Q_0)_\mu+\frac{1}{2} ((u_0)_{\alpha,\beta}+(u_0)_{\beta,\alpha})-\overline{H}_{\alpha\beta\lambda\mu}(N_0)_{\lambda\mu}=0, \text{ in } \Omega$$ and hence
$$(N_0)_{\alpha\beta}=H_{\alpha\beta\lambda \mu}\left(\frac{1}{2} ((u_0)_{\lambda,\mu}+(u_0)_{\lambda,\mu})+\frac{1}{2}(w_0)_\lambda(w_0)_\mu\right).$$
Considering the variation in $z^*$ we obtain
$$-\frac{\nabla^2 z^*_0}{K}-(H_{\alpha\beta\lambda\mu}((\tilde{M}_0)_{\lambda\mu}+z_0^* \delta_{\lambda\mu}))\delta_{\alpha\beta}=0, \text{ in } \Omega$$ so that
$$-\frac{\nabla^2 z^*_0}{K}-((w_0)_{11}+(w_0)_{22})=0,$$ that is, $$-\frac{\nabla^2 z^*_0}{K}-\nabla^2w_0=0,$$
and thus
$$z_0^*=-Kw_0, \text{ in } \Omega.$$

Finally, from the variation in $u$ we obtain

$$(\tilde{M}_0)_{\alpha\beta, \alpha\beta}-(Q_0)_{\alpha, \alpha}-P =0,$$
and
$$(N_0)_{\alpha\beta,\beta}+P_\alpha=0, \text{ in } \Omega,$$ so that $$v_0^* \in A^*.$$

From these last results and the Legendre transform properties we get
$$G_1^*(\tilde{M}_0,z_0^*)= \langle (w_0)_{\alpha\beta}, (\tilde{M}_0)_{\alpha\beta}+z^*_0 \delta_{\alpha\beta} \rangle_{L^2}-G_1(\{(w_0)_{\alpha\beta}\}),$$
$$G_2^*(N_0,Q_0)=\langle (u_0)_{\alpha,\beta},(N_0)_{\alpha\beta}\rangle_{L^2}+\langle (w_0)_{,\alpha},(Q_0)_\alpha \rangle_{L^2}-G_2(\{(u_0)_{\alpha,\beta},(w_0)_{,\alpha}\})$$
and
$$F^*(z^*_0)= -\langle \nabla w_0,\nabla z^*_0 \rangle_{L^2}-F(u_0).$$
From such results we may infer that
\begin{eqnarray}
J^*(v_0^*,z_0^*)&=&-G_1^*(\tilde{M}_0,z_0^*)-G_2^*(N_0,Q_0)+F^*(z^*_0)
\nonumber \\ &=& -\langle (w_0)_{,\alpha\beta}, (\tilde{M}_0)_{\alpha\beta}+z^*_0 \delta_{\alpha\beta} \rangle_{L^2}+G_1(\{(w_0)_{,\alpha\beta}\})
\nonumber \\ &&-\langle (u_0)_{\alpha,\beta},(N_0)_{\alpha\beta}\rangle_{L^2}-\langle (w_0)_{,\alpha},(Q_0)_\alpha \rangle_{L^2}+G_2(\{(u_0)_{\alpha,\beta},(w_0)_{,\alpha}\}) \nonumber \\ &&-\langle \nabla w_0,\nabla z^*_0 \rangle_{L^2}-F(u_0)
\nonumber \\ &=& -\langle w_0,P\rangle_{L^2}-\langle (u_0)_\alpha,P_\alpha \rangle_{L^2} \nonumber \\ &&+
\frac{1}{2}\int_\Omega h_{\alpha\beta\lambda\mu} \kappa_{\alpha\beta}(u_0) \kappa_{\lambda \mu}(u_0)\;dx
\nonumber \\ &&+\frac{1}{2}\int_\Omega H_{\alpha\beta\lambda\mu} \gamma_{\alpha\beta}(u_0) \gamma_{\lambda \mu}(u_0)\;dx
\nonumber \\ &=& J(u_0).
\end{eqnarray}

Also $$(\overline{h}_{11\lambda \mu}((\tilde{M}_0)_{\lambda\mu}+z^*_0\delta_{\lambda \mu}))_{,22}-(\overline{h}_{22\lambda \mu}((\tilde{M}_0)_{\lambda\mu}+z^*_0\delta_{\lambda \mu})_{,11}=(w_0)_{1122}-(w_0)_{2211}=0,$$ in distributional sense,
so that
$$J^*(v_0^*,z_0^*)=J_1^*(v_0^*,z_0^*)=J(u_0).$$

Now observe that
\begin{eqnarray}
J^*(v_0^*,z_0^*)&=&-G_1^*(\tilde{M}_0,z_0^*)-G_2^*(N_0,Q_0)+F^*(z^*_0)
\nonumber \\ &\leq& -\langle w_{,\alpha\beta}, (\tilde{M}_0)_{\alpha\beta}+z^*_0 \delta_{\alpha\beta} \rangle_{L^2}+G_1(\{w_{,\alpha\beta}\})
\nonumber \\ &&-\langle u_{\alpha,\beta},(N_0)_{\alpha\beta}\rangle_{L^2}+\langle w_{,\alpha},(Q_0)_\alpha \rangle_{L^2}+G_2(\{u_{\alpha,\beta},w_{,\alpha}\}) \nonumber \\ &&+F^*(z_0^*)\nonumber \\ &=&
\nonumber \\ &&
-\langle w,P\rangle_{L^2}-\langle u_\alpha,P_\alpha \rangle_{L^2} \nonumber \\ &&+G_1(\{w_{,\alpha\beta}\})+G_2(\{u_{\alpha,\beta},w_{,\alpha}\}) -F(u)+F(u)+\langle \nabla w, \nabla z_0^* \rangle_{L^2}+F^*(z_0^*) \nonumber \\ &=&
-\langle w,P\rangle_{L^2}-\langle u_\alpha,P_\alpha \rangle_{L^2} \nonumber \\ &&+G_1(\{w_{,\alpha\beta}\})+G_2(\{u_{\alpha,\beta},w_{,\alpha}\}) -F(u) \nonumber \\ &&+\frac{K}{2}\int_\Omega |\nabla w|^2\;dx-K\langle \nabla w, \nabla w_0 \rangle_{L^2}+\frac{K}{2}\int_\Omega |\nabla w_0|^2\;dx \nonumber \\ &=&
J(u)+\frac{K}{2}\int_\Omega |\nabla w-\nabla w_0|^2\;dx, \; \forall u \in U.
\end{eqnarray}
Therefore we have got
\begin{equation}\label{us356} J(u_0)=J^*(v_0^*,z_0^*)\leq \inf_{u \in U} \left\{J(u)+\frac{K}{2}\int_\Omega |\nabla w-\nabla w_0|^2\;dx\right\},\end{equation}
so that, from this we may infer that
\begin{eqnarray}
J(u_0)&=& \inf_{ u \in U}\left\{ J(u)+\frac{K}{2}\int_\Omega |\nabla w-\nabla w_0|^2\;dx\right\} \nonumber \\ &=&
J^*(v_0^*,z_0^*)  \nonumber \\ &=& \sup_{v^* \in A^*} J^*(v^*,z_0^*) \nonumber \\ &=& J_1^*(v_0^*,z_0^*).
\end{eqnarray}
The proof is complete.
\end{proof}

Our final result in this section refers to a concave dual variational formulation for the plate model in question.
\begin{thm} Denoting $v_0^*=(N_0,Q_0,\tilde{M}_0),$ suppose $(v_0^*,z_0^*,u_0) \in B^* \times Y^*\times U$ is such that
$$\delta J_3^*(v_0^*,z_0^*,u_0)=\mathbf{0}.$$
Assume $K>0$ is such that $$J_2^*(z^*)>0, \; \forall z^* \in Y^* \text{ such that } z^* \neq \mathbf{0}.$$
Under such hypotheses we have, $$\delta J(u_0)=\mathbf{0},$$ $$v_0^*=(N_0,Q_0, \tilde{M}_0) \in A^*$$ and
\begin{eqnarray} J(u_0)&=& \inf_{u \in U} J(u) \nonumber \\ &=& \sup_{v^* \in A^*}\left\{\inf_{z^* \in Y^*}J_1^*(v^*,z^*) \right\}
\nonumber \\ &=& J_1^*(v_0^*,z_0^*).
\end{eqnarray}
\end{thm}
\begin{proof}
The proof that $\delta J(u_0)=\mathbf{0}$ and $v_0^* \in A^*$ may be done exactly as in the proof of the last theorem.
First recall that we have assumed  $J_2^*(z^*)>0$ for all $z^* \in Y^*$ such that $z^* \neq \mathbf{0}.$

Therefore, denoting $J_4^* : B^* \times Y^* \times U \rightarrow \mathbb{R}$ by
\begin{eqnarray}J_4^*(v^*,z^*,u)&=&J_1^*(v^*,z^*)
\nonumber \\ &&+\langle w, \tilde{M}_{\alpha\beta, \alpha\beta}-Q_{\alpha, \alpha}-P \rangle_{L^2}
\nonumber \\ &&-\langle u_\alpha, N_{\alpha\beta,\beta}+P_\alpha \rangle_{L^2},
\end{eqnarray}
 since $J_1^*$ is concave in $v^*$ and convex in $z^*$, we have got
\begin{eqnarray}
J(u_0)&=&J^*_1(v_0^*,z_0^*)\nonumber \\ &=& \sup_{v^* \in A^*}\left\{ \inf_{z^* \in Y^*} J_1^*(v^*,z^*)\right\} \nonumber \\ &\leq&
\sup_{v^* \in A^*}J_1^*(v^*,z^*)
\nonumber \\ &\leq& \sup_{v^* \in B^*}J_4^*(v^*,z^*,u), \; \forall u \in U,\; z^* \in Y^*.\end{eqnarray}

From this we get
\begin{eqnarray}
J(u_0)&=&J^*_1(v_0^*,z_0^*)
\nonumber \\ &\leq& \inf_{z^* \in Y^*}\left\{\sup_{v^* \in B^*}J_4^*(v^*,z^*,u)\right\}, \; \forall u \in U.\end{eqnarray}

Now observe that denoting $$L(v^*,z^*)=(\overline{h}_{11\lambda \mu}(\tilde{M}_{\lambda \mu}+z^*\delta_{\lambda \mu}))_{,22}+
-(\overline{h}_{22\lambda \mu}(\tilde{M}_{\lambda \mu}+z^*\delta_{\lambda \mu}))_{,11},$$ from the variations in  $\tilde{M}_{11}$ and $\tilde{M}_{22}$, since the operator $C_0$ is self-adjoint, we may obtain,
\begin{eqnarray}&&-\overline{h}_{11\lambda \mu}(\tilde{M}_{\lambda \mu}+z^*\delta_{\lambda \mu})+w_{,11}
\nonumber \\ && +\{\overline{h}_{1111}[C_0[ L(v^*,z^*)]]\}_{,22}-\{\overline{h}_{2211}[C_0[ L(v^*,z^*)]]\}_{,11}=0,\end{eqnarray}
and
\begin{eqnarray}&&-\overline{h}_{22\lambda \mu}(\tilde{M}_{\lambda \mu}+z^*\delta_{\lambda \mu})+w_{,22}
\nonumber \\ && +\{\overline{h}_{1122}[C_0[ L(v^*,z^*)]]\}_{,22}-\{\overline{h}_{2222}[C_0[ L(v^*,z^*)]]\}_{,11}=0.\end{eqnarray}

From this, recalling that $$L(v^*,z^*)=(\overline{h}_{11\lambda \mu}(\tilde{M}_{\lambda \mu}+z^*\delta_{\lambda \mu}))_{,22}-(\overline{h}_{22\lambda \mu}(\tilde{M}_{\lambda \mu}+z^*\delta_{\lambda \mu}))_{,11}$$ we get

\begin{eqnarray}&&-(\overline{h}_{11\lambda \mu}(\tilde{M}_{\lambda \mu}+z^*\delta_{\lambda \mu}))_{,22}+
(\overline{h}_{22\lambda \mu}(\tilde{M}_{\lambda \mu}+z^*\delta_{\lambda \mu}))_{,11} \nonumber \\ &&
+\left\{\overline{h}_{1111}[C_0[ L(v^*,z^*)]]_{,22}-\overline{h}_{2211}[C_0[ L(v^*,z^*)]]_{,11}\right\}_{,22} \nonumber \\ &&
-\left\{\overline{h}_{1122}[C_0[ L(v^*,z^*)]]_{,22}-\overline{h}_{2222}[C_0[ L(v^*,z^*)]]_{,11}\right\}_{,11} \nonumber \\ &=&
-w_{1122}+w_{2211}=0,\end{eqnarray}
so that through the equation related to the  variation in $\tilde{M}$ satisfied we have obtained  $$L(v^*,z^*)=0,\; \forall v^* \in B^*, z^* \in Y^*.$$

From such results we may infer that if $$\{N_{\alpha\beta}(u) +K\delta_{\alpha\beta}\}$$ is positive definite  $\text{in } \Omega$, then $$\inf_{z^* \in Y^*}\left\{\sup_{v^* \in B^*}J_4^*(v^*,z^*,u)\right\}=J(u),$$ where $$N_{\alpha\beta}(u)=H_{\alpha\beta\lambda\mu} \gamma_{\lambda\mu}(u).$$

Therefore,
since $J_1^*$ is concave in $v^*$ and convex in $z^*$, from these last results and from the min-max theorem we have
\begin{eqnarray}
J(u_0)&=&J^*_1(v_0^*,z_0^*) \nonumber \\ &=& \sup_{v^* \in A^*}\left\{ \inf_{z^* \in Y^*} J_1^*(v^*,z^*)\right\} \nonumber \\ &\leq&  \inf_{z^* \in Y^*}\left\{\sup_{v^* \in B^*}J_4^*(v^*,z^*,u)\right\} \nonumber \\ &\leq& J(u), \; \forall u \in U.\end{eqnarray}

Thus,
\begin{eqnarray} J(u_0)&=& \inf_{u \in U} J(u) \nonumber \\ &=& \sup_{v^* \in A^*}\left\{\inf_{z^* \in Y^*}J_1^*(v^*,z^*) \right\}
\nonumber \\ &=& J_1^*(v_0^*,z_0^*).
\end{eqnarray}

The proof is complete.

\end{proof}

\section{An auxiliary theoretical result in analysis}

In this section we state and prove some theoretical results in analysis which will be used in the
subsequent sections.

\begin{thm} Let $\Omega \subset \mathbb{R}^3$ be an open, bounded and connected set with a regular (Lipschitzian)
boundary denoted by $\partial \Omega.$

Assume $\{u_n\} \subset W^{1,4}(\Omega)$ be such that $$\|u_n\|_{1,4} \leq K, \; \forall n \in \mathbb{N},$$
for some $K>0.$

Under such hypotheses there exists $u_0 \in W^{1,4}(\Omega) \cap C(\overline{\Omega})$ such that, up to a not relabeled subsequence,
$$u_n \rightharpoonup u_0, \text{ weakly in } W^{1,4}(\Omega),$$
$$u_n \rightarrow u_0 \text{ uniformly in } \overline{\Omega}$$
and
$$u_n \rightarrow u_0, \text{ strongly in } W^{1,3}(\Omega).$$
\end{thm}
\begin{proof} Since $W^{1,4}(\Omega)$ is reflexive, from the Kakutani and Sobolev Imbedding  theorems, up to a not relabeled
there exists $u_0 \in W^{1,4}(\Omega)$ such that
$$u_n \rightharpoonup u_0, \text{ weakly in } W^{1,4}(\Omega),$$
and
$$u_n \rightarrow u_0, \text{ strongly in } L^4(\Omega).$$

From the Rellich-Kondrachov Theorem, since for $m=1$, $p=4$ and $n=3$, we have $mp>n$, the following imbedding
is compact,
$$W^{1,4}(\Omega) \hookrightarrow C(\overline{\Omega}).$$

Thus, $$\{u_n\} \subset C(\overline{\Omega}),$$ and again up to a not relabeled subsequence,
$$u_n \rightarrow u_0 \text{ uniformly in } \overline{\Omega},$$
and also $$u_0 \in C(\overline{\Omega}),$$ so that
$$u_0 \in W^{1,4}(\Omega) \cap C(\overline{\Omega}).$$

Let $\varepsilon>0$. Hence, there exists $n_0 \in \mathbb{N}$ such that if $n>n_0$, then
$$|u_n(x)-u_0(x)|< \varepsilon, \text{ for almost all } x \in \Omega.$$

Let $$\varphi \in C^1_c(\Omega).$$ Choose $j \in \{1,2,3\}.$

Therefore, we may obtain

\begin{eqnarray}
\left|\left\langle \frac{\partial u_n }{\partial x_j}-\frac{\partial u_0 }{\partial x_j}, \varphi \right\rangle_{L^2}\right|
&=& \left|\left\langle u_n-u_0, \frac{\partial \varphi }{\partial x_j} \right\rangle_{L^2}\right|
\nonumber \\ &\leq&  \left\langle |u_n-u_0|, \left|\frac{\partial \varphi }{\partial x_j}\right| \right\rangle_{L^2}
\nonumber \\ &\leq& \varepsilon \left\|\frac{\partial \varphi }{\partial x_j}\right\|_1,\; \forall n >n_0.\end{eqnarray}

From this we may infer that
$$\lim_{n \rightarrow \infty}\left\langle \frac{\partial u_n }{\partial x_j}-\frac{\partial u_0 }{\partial x_j}, \varphi \right\rangle_{L^2}=0, \forall \varphi \in C^1_c(\Omega).$$

At this point we claim that
$$\lim_{n \rightarrow \infty}\left\langle \frac{\partial u_n }{\partial x_j}-\frac{\partial u_0 }{\partial x_j}, \varphi \right\rangle_{L^2}=0, \forall \varphi \in C_c(\Omega).$$

To prove such a claim, let $\varphi \in C_c(\Omega).$

Let a new $\varepsilon>0$ be given.

Hence, for each $r>0$ there exists $n_r \in \mathbb{N}$ such that if $n>n_r$, then $$\|u_n-u_0\|_\infty < \varepsilon r.$$

Observe that by density, we may obtain $\varphi_1 \in C^1_c(\Omega)$ such that
$$\|\varphi-\varphi_1\|_\infty< \varepsilon.$$

Hence,
\begin{eqnarray}
&&\left|\left\langle \frac{\partial u_n }{\partial x_j}-\frac{\partial u_0 }{\partial x_j}, \varphi \right\rangle_{L^2}\right|
\nonumber \\ &\leq& \left\langle \left|\frac{\partial u_n }{\partial x_j}-\frac{\partial u_0 }{\partial x_j}\right|, |\varphi -\varphi_1|\right\rangle_{L^2} \nonumber \\ &&+  \left|\left\langle\frac{\partial u_n }{\partial x_j}-\frac{\partial u_0 }{\partial x_j}, \varphi_1 \right\rangle_{L^2}\right| \nonumber \\ &\leq&
 \int_\Omega \left|\frac{\partial u_n }{\partial x_j}-\frac{\partial u_0 }{\partial x_j}\right|\;dx\|\varphi-\varphi_1\|_\infty \nonumber \\ &&+\left|\left\langle u_n-u_0,\frac{\partial \varphi_1 }{\partial x_j} \right\rangle_{L^2}\right| \nonumber \\ &\leq& 2 K_1 \varepsilon +\varepsilon\left\|\frac{\partial  \varphi_1}{\partial x_j}\right\|_1 \nonumber \\ &=& \left(2K_1+\left\|\frac{\partial  \varphi_1}{\partial x_j}\right\|_1\right)\varepsilon, \forall n>n_1.
 \end{eqnarray}
where $K_1>0$ is such that $$\|u_0\|_1<K_1, \|u_0\|_2<K_1$$ and $$\|u_n\|_1 < K_1,\;\|u_n\|_2 < K_1\; \forall n \in \mathbb{N}.$$

From this we may infer that
\begin{equation}\label{eq29}\lim_{n \rightarrow \infty}\left\langle \frac{\partial u_n }{\partial x_j}-\frac{\partial u_0 }{\partial x_j}, \varphi \right\rangle_{L^2}=0,\; \forall \varphi \in C_c(\Omega)\end{equation} so that
the claim holds.

Since $\Omega$ is bounded, we have $W^{1,4}(\Omega) \subset W^{1,2}(\Omega)$.

From the Gauss-Green Formula for such a latter space, we obtain

\begin{eqnarray}
&&\lim_{n \rightarrow \infty }\left|\frac{\partial u_n(x) }{\partial x_j}-\frac{\partial u_0(x) }{\partial x_j}\right|
\nonumber \\ &=&\lim_{ n \rightarrow \infty }\left( \lim_{r \rightarrow 0^+} \frac{\left| \int_{B_r(x)}\left( \frac{\partial u_n(y) }{\partial x_j}-\frac{\partial u_0(y) }{\partial x_j}\right)\;dy\right|}{m(B_r(x))}\right) \nonumber \\ &\leq&\limsup_{ n \rightarrow \infty }\left( \limsup_{r \rightarrow 0^+} \frac{\left| \int_{B_r(x)}\left( \frac{\partial u_n(y) }{\partial x_j}-\frac{\partial u_0(y) }{\partial x_j}\right)\;dy\right|}{m(B_r(x))}\right) \nonumber \\ &=&\limsup_{r \rightarrow 0^+}\left(\limsup_{ n \rightarrow \infty }  \frac{\left| \int_{B_r(x)}\left( \frac{\partial u_n(y) }{\partial x_j}-\frac{\partial u_0(y) }{\partial x_j}\right)\;dy\right|}{m(B_r(x))}\right) \nonumber \\ &=&
\limsup_{r \rightarrow 0^+}\left(\limsup_{ n \rightarrow \infty }\frac{\left| \int_{\partial B_r(x)} (u_n(y)-u_0(y))\nu_i \;dS(y)\right|}{m(B_r(x))}\right) \nonumber \\ &=&
\limsup_{r \rightarrow 0^+}\left(\limsup_{ n \rightarrow \infty }\frac{\left| (u_n(\tilde{y})-u_0(\tilde{y}))\int_{\partial B_r(x)} \nu_i \;dS(y)\right|}{m(B_r(x))}\right) \nonumber \\ &\leq& \varepsilon \limsup_{r \rightarrow 0^+} \frac{ \int_{\partial B_r(x)} r |\nu_i| \;dS(y)}{m(B_r(x))}
 \nonumber \\ &\leq& K_1 \varepsilon, \text{ for almost all } x \in \Omega,\end{eqnarray}
where $\tilde{y} \in \overline{B_r(x)}$ depends on $r$ and $n$.

Therefore, we may infer that
$$\lim_{n \rightarrow \infty}\frac{\partial u_n(x) }{\partial x_j}=\frac{\partial u_0(x) }{\partial x_j}, \text{ a. e. in }  \Omega.$$

Here we define
$$A_{n, \varepsilon}=\left\{ x \in \Omega \;:\;\left|\frac{\partial u_n(x) }{\partial x_j}-\frac{\partial u_0(x) }{\partial x_j}\right|< \varepsilon\right\}.$$

Define also
$$B_n=\cap_{k=n}^\infty A_{k,\varepsilon}.$$

Observe that for almost all $x \in \Omega$, there exists $n_x \in \mathbb{N}$ such that if $n>n_x$, then
$$\left|\frac{\partial u_n(x) }{\partial x_j}-\frac{\partial u_0(x) }{\partial x_j}\right|< \varepsilon,$$
so that almost all $x \in B_n,\; \forall n>n_x.$

From this $$\Omega = \left(\cup_{n=1}^\infty B_n\right) \cup B_0,$$
where $m(B_0)=0.$

Also $$\cup_{k=1}^n B_k=B_n,$$
so that
$$\lim_{ n \rightarrow \infty } m(B_n)=m(\Omega).$$

Observe that there exists $n_0 \in \mathbb{N}$ such that if $n>n_0$, then
$$\sqrt[4]{m(\Omega \setminus B_n)}< \varepsilon/K^3.$$

Consequently fixing $n >n_0$, from the generalized H\"{o}lder inequality, if $m>n$, we have
\begin{eqnarray}&&\int_\Omega \left|\frac{\partial u_m }{\partial x_j}-\frac{\partial u_0 }{\partial x_j}\right|^3 \;dx
\nonumber \\ &=& \int_{\Omega \setminus B_n} \left|\frac{\partial u_m }{\partial x_j}-\frac{\partial u_0 }{\partial x_j}\right|^3\;dx
\nonumber \\ &&+ \int_{B_n}\left| \frac{\partial u_m }{\partial x_j}-\frac{\partial u_0 }{\partial x_j}\right|^3\;dx \nonumber \\ &\leq& \left\|\frac{\partial u_m }{\partial x_j}-\frac{\partial u_0 }{\partial x_j}\right\|_4^3 \|\chi_{\Omega\setminus B_n}\|_4 +\varepsilon^3 m(\Omega) \nonumber \\ &\leq&
\varepsilon+ \varepsilon^3 m(\Omega).
\end{eqnarray}

Summarizing, we may infer that
$$\int_\Omega \left|\frac{\partial u_m }{\partial x_j}-\frac{\partial u_0 }{\partial x_j}\right|^3 \;dx \rightarrow 0, \text{ as } m \rightarrow \infty, \; \forall j \in \{1,2,3\}.$$
so that
$$u_n \rightarrow u_0, \text{ strongly in } W^{1,3}(\Omega).$$

The proof is complete.

\end{proof}
\section{An existence result for a model in elasticity}
In this section we present an existence result for a non-linear elasticity model.

\begin{thm}
Let $\Omega \subset \mathbb{R}^3$ be an open, bounded and connected set with a regular (Lipschitzian)
boundary denoted by $\partial \Omega.$

Consider the functional $J:U \rightarrow \mathbb{R}$ defined by
\begin{eqnarray}
J(u)&=& \frac{1}{2}\int_\Omega H_{ijkl}\left( \frac{u_{i,j}+u_{j,i}}{2}+\frac{u_{m,i}u_{m,j}}{2}\right)\left( \frac{u_{k,l}+u_{l,k}}{2}+\frac{u_{p,k}u_{p,l}}{2}\right)\;dx \nonumber \\ &&-\langle P_i,u_i \rangle_{L^2},
\end{eqnarray}
where $U=W_0^{1,4}(\Omega;\mathbb{R}^3)$, $P_i \in L^\infty(\Omega),\;\forall i \in \{1,2,3\}.$

Moreover, $\{H_{ijkl}\}$ is a fourth order constant tensor such that
$$H_{ijkl}t_{ij}t_{kl} \geq c_0 t_{ij}t_{ij}, \; \forall \text{ symmetric } t \in \mathbb{R}^{2 \times 2}$$
and
$$H_{ijkl}t_{mi}t_{mj}t_{kp}t_{lp} \geq c_1 \sum_{i,j=1}^3 t_{ij}^4,\; \forall \text{ symmetric } t \in \mathbb{R}^{2\times 2},$$
for some real constants $c_0>0, c_1>0$.

Under such hypotheses, there exists $u_0 \in U$ such that
$$J(u_0)=\min_{u \in U} J(u).$$
\end{thm}
\begin{proof}
First observe that we may find a positive definite tensor $\{T_{ij}\} \subset L^\infty(\Omega;\mathbb{R}^{2\times 2})$ such that
$$T_{ij,j}+P_i=0, \text{ in } \Omega.$$

Hence, denoting
$$v_{ij}(u)=\frac{u_{i,j}+u_{j,i}}{2}+\frac{u_{m,i}u_{m,j}}{2},$$
we have
$$\frac{u_{i,j}+u_{j,i}}{2}=v_{ij}(u)-\frac{u_{m,i}u_{m,j}}{2},$$
so that
\begin{eqnarray}
J(u)&=&\frac{1}{2}\int_\Omega H_{ijkl}v_{ij}(u)v_{kl}(u)\;dx+\langle T_{ij,j},u_i\rangle_{L^2} \nonumber \\ &=&
\frac{1}{2}\int_\Omega H_{ijkl}v_{ij}(u)v_{kl}(u)\;dx-\left\langle T_{ij},\frac{u_{i,j}+u_{j,i}}{2}\right\rangle_{L^2} \nonumber \\ &=&\frac{1}{2}\int_\Omega H_{ijkl}v_{ij}(u)v_{kl}(u)\;dx-\left\langle T_{ij},v_{ij}(u)-\frac{u_{m,i}\;u_{m,j}}{2}\right\rangle_{L^2} \nonumber \\ &=&\frac{1}{2}\int_\Omega H_{ijkl}v_{ij}(u)v_{kl}(u)\;dx-\left\langle T_{ij},v_{ij}(u)\right\rangle_{L^2}+\left\langle T_{ij},\frac{u_{m,i}\;u_{m,j}}{2}\right\rangle_{L^2},\; \forall u \in U.
\end{eqnarray}
From this and the hypotheses on $\{H_{ijkl}\}$ it is clear that $J$ is bounded below so that there exists $\alpha \in \mathbb{R}$ such that
$$\alpha=\inf_{u \in U} J(u).$$

Let $\{u_n\} \subset U$ be a minimizing sequence for $J$, that is, let such a sequence be such that
$$J(u_n) \rightarrow \alpha, \; \text{ as } n \rightarrow \infty.$$

Also from the hypotheses on $\{H_{ijkl}\}$ and the Poincar\'{e} inequality, we have that there exists $K>0$ such that
$$ \|u_n\|_{1,4} \leq K,\; \forall n \in \mathbb{N}.$$

From the auxiliary result in the last section, there exists $u_0 \in C^0(\overline{\Omega};\mathbb{R}^3) \cap W^{1,4}(\Omega;\mathbb{R}^3)$ such that, up to a not relabeled subsequence,
$$u_n \rightarrow u_0, \text{ strongly in } W^{1,3}(\Omega:\mathbb{R}^3).$$

From such a latter result, up to a not relabeled subsequence, we may obtain
$$\frac{(u_n)_{i,j}+(u_n)_{j,i}}{2}+\frac{(u_n)_{m,i}(u_n)_{m,j}}{2} \rightharpoonup \frac{(u_0)_{i,j}+(u_0)_{j,i}}{2}+\frac{(u_0)_{m,i}(u_0)_{m,j}}{2}, \text{ weakly in } L^{3/2}(\Omega).$$

Since $L^{3/2}(\Omega)$ is reflexive, from the convexity of $J$ in $v_{ij}(u)$ and since $\{T_{ij}\}$ is positive definite, we have that
$$\alpha=\liminf_{n \rightarrow \infty}J(u_n) \geq J(u_0),$$
so that
$$J(u_0)=\min_{u \in U} J(u).$$

The proof is complete.

\end{proof}

\section{Another existence result for a model in elasticity}
In this section we present another existence result  for a similar (to the previous one)  non-linear elasticity model.

\begin{thm}
Let $\Omega \subset \mathbb{R}^3$ be an open, bounded and connected set with a regular (Lipschitzian)
boundary denoted by $\partial \Omega.$

Consider the functional $J:U \rightarrow \mathbb{R}$ defined by
\begin{eqnarray}
J(u)&=& \frac{1}{2}\int_\Omega H_{ijkl}\left( \frac{u_{i,j}+u_{j,i}}{2}+\frac{u_{m,i}u_{m,j}}{2}\right)\left( \frac{u_{k,l}+u_{l,k}}{2}+\frac{u_{p,k}u_{p,l}}{2}\right)\;dx \nonumber \\ &&-\langle P_i,u_i \rangle_{L^2}-\langle P^t_i,u_i \rangle_{L^2(\Gamma_t)},
\end{eqnarray}
where $$\partial \Omega =\Gamma=\Gamma_0 \cup \Gamma_t,$$ $$\Gamma_0 \cap \Gamma_t=\emptyset,$$
$m_\Gamma(\Gamma_0)>0,$ $m_\Gamma(\Gamma_t)>0,$ $P_i \in L^\infty(\Omega)\cap W^{1,2}(\Omega),\; P^t_i \in L^\infty(\Gamma_t),\;
\forall i \in \{1,2,3\}.$

Moreover $$U=\{u \in W^{1,4}(\Omega;\mathbb{R}^3)\;: \; u=\hat{u}_0 \text{ on } \Gamma_0\},$$
where we assume $\hat{u}_0 \in W^{1,4}(\Omega).$

Furthermore, $\{H_{ijkl}\}$ is a fourth order symmetric constant tensor such that
$$H_{ijkl}t_{ij}t_{kl} \geq c_0 t_{ij}t_{ij}, \; \forall \text{ symmetric } t \in \mathbb{R}^{2 \times 2}$$
and
$$H_{ijkl}t_{mi}t_{mj}t_{kp}t_{lp} \geq c_1 \sum_{i,j=1}^3 t_{ij}^4,\; \forall \text{ symmetric } t \in \mathbb{R}^{2\times 2},$$
for some real constants $c_0>0, c_1>0$.

Under such hypotheses, there exists $u_0 \in U$ such that
$$J(u_0)=\min_{u \in U} J(u).$$
\end{thm}
\begin{proof}
First observe that we may find a positive definite tensor $\{T_{ij}\} \subset L^\infty(\Omega;\mathbb{R}^{2\times 2})\cap W^{1,2}(\Omega;\mathbb{R}^{2\times 2})$ such that
$$T_{ij,j}+P_i=0, \text{ in } \Omega.$$

Hence, denoting
$$v_{ij}(u)=\frac{u_{i,j}+u_{j,i}}{2}+\frac{u_{m,i}u_{m,j}}{2},$$
we have
$$\frac{u_{i,j}+u_{j,i}}{2}=v_{ij}(u)-\frac{u_{m,i}u_{m,j}}{2},$$
so that, from this, the Gauss-Green formula and the Trace Theorem,
\begin{eqnarray}
J(u)&=&\frac{1}{2}\int_\Omega H_{ijkl}v_{ij}(u)v_{kl}(u)\;dx+\langle T_{ij,j},u_i\rangle_{L^2}-\langle P_i,u_i \rangle_{L^2}-\langle P^t_i,u_i \rangle_{L^2(\Gamma_t)} \nonumber \\ &=&
\frac{1}{2}\int_\Omega H_{ijkl}v_{ij}(u)v_{kl}(u)\;dx-\left\langle T_{ij},\frac{u_{i,j}+u_{j,i}}{2}\right\rangle_{L^2} \nonumber \\ &&+\langle u_i, T_{ij}\nu_j\rangle_{L^2(\Gamma_t)}+\langle (u_0)_i, T_{ij}\nu_j\rangle_{L^2(\Gamma_0)}-\langle P_i,u_i \rangle_{L^2}-\langle P^t_i,u_i \rangle_{L^2(\Gamma_t)} \nonumber \\ &=&\frac{1}{2}\int_\Omega H_{ijkl}v_{ij}(u)v_{kl}(u)\;dx-\left\langle T_{ij},v_{ij}(u)-\frac{u_{m,i}\;u_{m,j}}{2}\right\rangle_{L^2}\nonumber \\ &&+\langle u_i, T_{ij}\nu_j\rangle_{L^2(\Gamma_t)}+\langle (\hat{u}_0)_i, T_{ij}\nu_j\rangle_{L^2(\Gamma_0)}-\langle P_i,u_i \rangle_{L^2}-\langle P^t_i,u_i \rangle_{L^2(\Gamma_t)} \nonumber \\ &\geq&\frac{1}{2}\int_\Omega H_{ijkl}v_{ij}(u)v_{kl}(u)\;dx-\langle T_{ij},v_{ij}(u)\rangle_{L^2}+\left\langle T_{ij},\frac{u_{m,i}\;u_{m,j}}{2}\right\rangle_{L^2}\nonumber \\ &&-K_3\sum_{i=1}^3\|u_i\|_{1,4}-K_3\|\hat{u}_0\|_{1,4},\; \forall u \in U,
\end{eqnarray}
for some appropriate $K_3>0$.

From this, the hypotheses on $\{H_{ijkl}\}$ and a Poincar\'{e} type inequality, since $\{T_{ij}\}$ is positive definite, it is clear that $J$ is bounded below so that there exists $\alpha \in \mathbb{R}$ such that
$$\alpha=\inf_{u \in U} J(u).$$

Let $\{u_n\} \subset U$ be a minimizing sequence for $J$, that is, let such a sequence be such that
$$J(u_n) \rightarrow \alpha, \; \text{ as } n \rightarrow \infty.$$

Also from the hypotheses on $\{H_{ijkl}\}$ and a Poincar\'{e} type inequality, we have that there exists $K>0$ such that
$$ \|u_n\|_{1,4} \leq K,\; \forall n \in \mathbb{N}.$$

From the auxiliary result in the last section, there exists $u_0 \in C^0(\overline{\Omega};\mathbb{R}^3) \cap W^{1,4}(\Omega;\mathbb{R}^3)$ such that, up to a not relabeled subsequence,
$$u_n \rightarrow u_0, \text{ strongly in } W^{1,3}(\Omega;\mathbb{R}^3).$$

From such a latter result, up to a not relabeled subsequence, we may obtain
$$\frac{(u_n)_{i,j}+(u_n)_{j,i}}{2}+\frac{(u_n)_{m,i}(u_n)_{m,j}}{2} \rightharpoonup \frac{(u_0)_{i,j}+(u_0)_{j,i}}{2}+\frac{(u_0)_{m,i}(u_0)_{m,j}}{2}, \text{ weakly in } L^{3/2}(\Omega).$$

 Also from the continuity of the Trace operator we get $$u_0= \hat{u}_0, \text{ on } \Gamma_0,$$ so that
 $u_0 \in U.$

Since $L^{3/2}(\Omega)$ is reflexive, from the convexity of $J$ in $\{v_{ij}(u)\}$ and since $\{T_{ij}\}$ is positive definite, we have that
$$\alpha=\liminf_{n \rightarrow \infty}J(u_n) \geq J(u_0),$$
so that,
$$J(u_0)=\min_{u \in U} J(u).$$

The proof is complete.

\end{proof}


\begin{thebibliography}{}
%
%
\bibitem{1}
R.A. Adams and J.F. Fournier, {Sobolev Spaces}, 2nd edn.
 (Elsevier, New York, 2003).

\bibitem{85} W.R. Bielski and J.J. Telega, A Contribution to Contact Problems for a Class of Solids and Structures,
Arch. Mech., 37, 4-5, pp. 303-320, Warszawa 1985.

\bibitem{2900} W.R. Bielski, A. Galka, J.J. Telega, The Complementary Energy Principle and Duality for
Geometrically Nonlinear Elastic Shells. I. Simple case of moderate rotations around a tangent to the middle surface.
Bulletin of the Polish Academy of Sciences, Technical Sciences, Vol. 38, No. 7-9, 1988.
\bibitem{15} F. Botelho, {\it Dual Variational Formulations for a
Non-linear Model of Plates}, Journal of Convex Analysis, 17 , No.
1, 131-158 (2010).

 \bibitem{120}
F.S. Botelho, {Functional Analysis and Applied Optimization in Banach Spaces},
 (Springer Switzerland, 2014).
 \bibitem{700} F.S. Botelho, Functional Analysis, Calculus of Variations and Numerical Methods in Physics and Engineering, CRC Taylor and Francis, Florida, 2020.


 \bibitem{903} P.Ciarlet, {\it Mathematical Elasticity}, {Vol. I -- Three Dimensional Elasticity}, North Holland, Elsevier (1988).

\bibitem{[3]} P.Ciarlet, {\it Mathematical Elasticity}, {Vol. II -- Theory of Plates}, North Holland Elsevier (1997).

\bibitem{[4]} P.Ciarlet, {\it Mathematical Elasticity}, {Vol. III -- Theory of Shells}, North Holland Elsevier (2000).

\bibitem{[6]} I.Ekeland and R.Temam, {\it Convex Analysis and
Variational Problems.} North Holland (1976).
\bibitem{10} J.J. Telega, {\it On the complementary energy principle in non-linear elasticity.
 Part I: Von Karman plates and three dimensional solids}, C.R. Acad.
 Sci. Paris, Serie II, 308, 1193-1198; Part II: Linear elastic solid
 and non-convex boundary condition. Minimax approach, ibid, pp.
 1313-1317 (1989)

\bibitem{12} J.F. Toland, {\it A duality principle for non-convex
optimisation and the calculus of variations}, Arch. Rath. Mech.
Anal., {\bf 71}, No. 1 (1979), 41-61.




\end{thebibliography}
\end{document}